\documentclass[psamsfonts]{amsart}

\usepackage{amssymb, latexsym, eucal}

\usepackage[utf8]{inputenc}

\usepackage{braket}

\usepackage{color}

\definecolor{darkgreen}{rgb}{0,.6,0}
\definecolor{darkred}{rgb}{.6,0,0}
\definecolor{darkblue}{rgb}{0,.2,.5}

\title[A max. Boolean sublattice not the range of a Banasch. function]{A maximal Boolean sublattice that is not the range of a Banaschewski function}

\author[S. Mokri\v{s}]{Samuel Mokri\v{s}}
\author[P. R\r{u}\v{z}i\v{c}ka]{Pavel R\r{u}\v{z}i\v{c}ka}

\address{
  Department of Algebra\\
	Faculty of mathematics and Physics\\
	Charles University in Prague\\
  Sokolovsk\'a 83\\
  186 75, Prague\\
  Czech republic
}
\email[S. Mokri\v{s}]{smokris@seznam.cz}
\email[P. R\r{u}\v{z}i\v{c}ka]{ruzicka@karlin.mff.cuni.cz}

\thanks{The first author was partially supported by the project SVV-2015-260227 of Charles University
in Prague. The second author was partially supported by the Grant Agency of the Czech Republic under 
the grant no. GACR 14-15479S}
  
\subjclass[2010]{06A15, 06C20, 06D75}

\keywords{lattice, complemented, modular, Boolean, Banaschewski function, Schmidt's construction, closure operator, adjunction}

\date{\today}

\dedicatory{Dedicated to J\'{a}ra Cimrman on the occasion of his 50th birthday.}

\numberwithin{equation}{section}

\theoremstyle{plain}                                     
\newtheorem{lemma}{Lemma}[section]
\newtheorem{theorem}[lemma]{Theorem}
\newtheorem{proposition}[lemma]{Proposition}
\newtheorem{corollary}[lemma]{Corollary}

\newtheorem*{problem}{Problem}

\theoremstyle{remark}

\newtheorem{remark}[lemma]{Remark}    

\theoremstyle{definition}
\newtheorem{definition}[lemma]{Definition}

\newcommand{\setof}[2]{\{ #1 \mid #2 \}}
\newcommand{\tuple}[1]{\left\langle #1 \right\rangle}        

\newcommand{\lattice}[1]{\ensuremath{{\boldsymbol{#1}}}}     
\newcommand{\vectorspace}[1]{\ensuremath{\boldsymbol{#1}}}   
\newcommand{\ring}[1]{\ensuremath{#1}}                       

\newcommand{\meansymbol}{\mu}
\newcommand{\mean}[1]{\ensuremath{\meansymbol{#1}}}   
\newcommand{\closure}[1]{\ensuremath{\overline{#1}}}  
\newcommand{\Sub}{\mathrm{Sub}}
\newcommand{\Span}{\mathrm{Span}}

\newcommand{\ABC}[1][]{\tuple{A#1, B#1, C#1}}
\newcommand{\abc}[1][]{\tuple{a#1, b#1, c#1}}
\newcommand{\AC}[1][]{\tuple{A#1, C#1}}
\newcommand{\kompl}[1]{\kappa \setminus {#1}}

\newcommand{\Pk}{\ensuremath{\mathcal{P} (\kappa)}}
\newcommand{\Fk}{\ensuremath{\mathcal{F} (\kappa)}}

\begin{document}

\begin{abstract} We construct a countable bounded sublattice of the lattice of all subspaces of a vector space with
two non-isomorphic maximal Boolean sublattice. We represent one of them as the range of a Banschewski function and we
prove that this is not the case of the other. Hereby we solve a problem of F. Wehrung. 
\end{abstract}

\maketitle

\section{Introduction}\label{Introduction} 

In \cite{Weh10} Friedrich Wehrung defined a \emph{Banaschewski function} on a bounded complemented lattice $\lattice{L}$ as an antitone (i.e., order-reversing) map sending each element of $\lattice{L}$ to its complement, being motivated by the earlier result of Bernhard Banaschewski that such a function exists on the lattice of all subspaces of a vector space \cite{Ban57}. Wehrung extended Banaschewski's result by proving that every countable complemented modular lattice has a Banaschewski function with a Boolean range and that all the possible ranges of Banaschewski functions on $\lattice{L}$ are isomorphic \cite[Corollary 4.8]{Weh10}.  

Still in \cite{Weh10} Wehrung defined a ring-theoretical analogue of Banaschewski function that, for a von Neuman regular ring $R$, is closely connected to the lattice-theoretical Banaschewski function on the lattice $\lattice{L}(\ring{R})$ of all finitely generated right ideals of $\ring{R}$. He made use of these ideas  to construct a unit-regular ring $\ring{S}$ (in fact of bounded index $3$) of size $\aleph_1$ with no Banaschewski function \cite{Weh11}. 
 
Furthermore in \cite{Weh10} Wehrung defined notions of a Banaschewski measure and a Banaschewski trace on sectionally complemented modular lattices and he proved that a sectionally complemented lattice which is either modular with a large $4$-frame or Arguesian with a large $3$-frame is coordinatizable (i.e. isomorphic to $\lattice{L}(\ring{R})$ for a possibly non-unital von Neumann regular ring $\ring{R}$) if and only if it has a Banaschewski trace. Applying this results, he constructed a non-coordinatizable sectionally complemented modular lattice, of size $\aleph_1$, with a large $4$-frame \cite[Theorem 7.5]{Weh10}. 

The aim of our paper is to solve the second problem from \cite{Weh10}:	

\begin{problem}[Problem 2 from \cite{Weh10}]\label{prob_2}
	Is every maximal Boolean sublattice of an at most countable complemented modular lattice $\lattice{L}$ the range of some Banaschewski function on $\lattice{L}$? Are any two such Boolean sublattices isomorphic?
\end{problem}

We construct a countable complemented modular lattice $\lattice{S}$ with two non-isomorphic maximal Boolean sublattices $\lattice{B}$ and $\lattice{E}$. We represent $\lattice{E}$ as the range of a Banaschewski function on $\lattice{S}$ and we prove that $\lattice{B}$ is not the range of any Banaschewski function. Finally we represent the lattice $\lattice{S}$ as a bounded sublattice of the subspace-lattice of a vector space.  

 \section{Basic concepts}\label{Basic concepts} 

We start with recalling same basic notions as well as the precise definition of the Banaschewski function adopted from \cite{Weh10}. Next we outline the
Schmidt's $\lattice{M}_3[\lattice{L}]$ construction, which we then apply to define the bounded modular lattice $\lattice{S}$ containing a pair of non-isomorphic maximal Boolean sublattices.   

\subsection{Some standard notions, notation, and the Banaschewski function} A lattice $\lattice{L}$ is \emph{bounded} if it has both the least element and the greatest element, denoted by $0_\lattice{L}$ and $1_\lattice{L}$, respectively. A \emph{bounded sublattice} of a bounded lattice is its sublattice containing the bounds. Given elements $a,b,c$ of a lattice $\lattice{L}$ with zero, we will use the notation $c = a \oplus b$ when $a \wedge b = 0_\lattice{L}$ and $a \vee b = c$. A \emph{complement} of an element $a$ of a bounded lattice $\lattice{L}$ is an element $a'$ of $\lattice{L}$ such that $a \oplus a' = 1_\lattice{L}$. A lattice $\lattice{L}$ is said to be \emph{complemented} provided that it is bounded and each element of $\lattice{L}$ has a (not necessarily unique) complement. A lattice $\lattice{L}$ is \emph{relatively complemented} if each of its interval is complemented. Note that a relatively complemented lattice is not necessarily bounded.

We say that a lattice $\lattice{L}$ is \emph{uniquely complemented} if it is bounded and each element of $\lattice{L}$ has a unique complement. By a \emph{Boolean lattice} we mean a lattice reduct of a Boolean algebra, that is, a distributive uniquely complemented lattice. For the clarity, let us recall the formal definition of the Banaschewski function \cite[Definition 3.1]{Weh10}:

\begin{definition} Let $\lattice{L}$ be a bounded lattice. A \emph{Banaschewski function} on $\lattice{L}$ is a map 
$f \colon \lattice{L}\to \lattice{L}$ such that both
	\begin{enumerate}
		\item $x \leq y$ implies $f(x) \geq f(y)$, for all $x,y \in \lattice{L}$, and 
		\item $f(x) \oplus x = 1_{\lattice{L}}$ for all $ x \in \lattice{L}$,
	\end{enumerate}
hold true.	
\end{definition}

\subsection{The $\lattice{M}_3[\lattice{L}]$-construction.} Let $\lattice{L}$ be a lattice. We will call a triple $\abc \in \lattice{L}^3$ {\em balanced}, if it satisfies 
\begin{equation*}
a \wedge b = a \wedge c = b \wedge c
\end{equation*} 
and we denote by $\lattice{M}_3[\lattice{L}]$ the set of all balanced triples. It is readily seen that $\lattice{M}_3[\lattice{L}]$ is a meet-subsemilattice of the cartesian product $\lattice{L}^3$. However, it is not necessarily a join-subsemilattice, for one easily observes that the join of balanced triples may not be balanced. The $\lattice{M}_3[\lattice{L}]$-construction was introduced by E. T. Schmidt \cite{Sch68,Sch74} for a bounded distributive lattices $\lattice{L}$. He proved \cite[Lemma 1]{Sch74} that in this case $\lattice{M}_3[\lattice{L}]$ is a bounded modular lattice and that it is a congruence-preserving extension of the distributive lattice $\lattice{L}$. This result was later extended by Gr\"atzer and Schmidt in various directions \cite{GS95,GS03}. In particular, in \cite{GS95} they proved that every lattice with a non-trivial distributive interval has a proper congruence-preserving extension. This was further improved by Gr\"atzer and Wehrung in \cite{GW99d}, where they introduced a modification of the $\lattice{M}_3[\lattice{L}]$-construction, called $\lattice{M}_3\langle\lattice{L}\rangle$-construction. Using this new idea they proved that every non-trivial lattice admits a proper congruence-preserving extension.     

 The lattice constructions $\lattice{M}_3[\lattice{L}]$ and $\lattice{M}_3\langle \lattice{L} \rangle$ appeared in the series of papers by Gr\"atzer and Wehrung \cite{GW99a,GW99b,GW99c,GW99d,GW99e,GW00,GW01b} dealing with semilattice tensor product and its related structures, namely the box product and the lattice tensor product \cite[Definition 2.1 and definition 3.3]{GW99c}. Indeed, $\lattice{M}_3 \boxtimes \lattice{L} \simeq \lattice{M}_3 \langle \lattice{L} \rangle$ for every lattice $\lattice{L}$ and $\lattice{M}_3 \otimes \lattice{L} \simeq \lattice{M}_3[\lattice{L}]$ whenever $\lattice{L}$ has zero and $\lattice{M}_3 \otimes \lattice{L}$ is a lattice (see \cite[Theorem 6.5]{GW01b} and \cite[Corollary 6.3]{GW99b}). In particular, the latter is satisfied when the lattice $\lattice{L}$ is modular with zero. Note also, that if $\lattice{L}$ is a bounded distributive lattice both the constructions $\lattice{M}_3[\lattice{L}]$ and $\lattice{M}_3\langle \lattice{L} \rangle$ coincide. In our paper we get by with this simple case. 

Let \lattice{L} be a distributive lattice. Given a triple $\abc \in \lattice{L}^3$, we define
\begin{equation*}
\mean{\abc} = (a \wedge b) \vee (a \wedge c) \vee (b \wedge c) 
\end{equation*}
and we set
\begin{equation}\label{equation:closure}
\closure{\abc} = \tuple{a \vee \mean{\abc}, b \vee \mean{\abc}, c \vee \mean{\abc}}.
\end{equation}
Using the distributivity of $\lattice{L}$ one easily sees that $\closure{\abc}$ is the least balanced triple $\geq \abc$ in $\lattice{L}^3$ and that the map $\closure{\tuple{-}} \colon \lattice{L}^3 \to \lattice{L}^3$ determines a closure operator on the lattice $\lattice{L}^3$ (see \cite[Lemma 2.3]{Weh10} 
for a refinement of this observation). It is also clear that
\begin{equation*}
\begin{array}{ll}
a \vee \mean{\abc} &= a \vee (b \wedge c), \\
b \vee \mean{\abc} &= b \vee (a \wedge c), \\
c \vee \mean{\abc} &= c \vee (a \wedge b).
\end{array}
\end{equation*}
A triple $\abc \in \lattice{L}^3$ is closed with respect to the closure operator if and only if it is balanced. Therefore the set of all balanced triples, denoted by $\lattice{M}_3[\lattice{L}]$, forms a lattice \cite[Lemma 2.1]{Weh10}, where
\begin{equation} \label{equation:join}
\abc \vee \abc['] = \closure{\tuple{a \vee a', b \vee b', c \vee c'}}  
\end{equation}
and
\begin{equation}\label{equation:meet} 
\abc \wedge \abc['] = \tuple{a \wedge a', b \wedge b', c \wedge c'}.  
\end{equation}
By \cite[Lemma 2.9]{GW99b} the lattice $\lattice{M}_3[\lattice{L}]$ is modular if and only if the lattice $\lattice{L}$ is distributive. The ``if'' part of the equivalence is included in the above mentioned \cite[Lemma 1]{Sch74}. 

\section{The lattice}\label{The lattice} 

Fix an infinite cardinal $\kappa$. As it is customary, we identify $\kappa$ with the set of all ordinals of cardinality less than $\kappa$. Let us denote 
by $\Pk$ the Boolean lattice of all subsets of $\kappa$ and set
\begin{equation*}
	\Fk := \setof{X \subseteq \kappa}{ X \text{ is finite or } \kappa \setminus X \text{ is finite}}.
\end{equation*}
It is well-known that $\Fk$ is a bounded Boolean sublattice of $\Pk$. Next, let us define
\begin{equation*}
  \lattice{T} = \setof{ \ABC \in \Fk^3 }{ C \setminus \mean{\ABC} \text{ is finite}}.
\end{equation*}  
 
\begin{lemma}\label{lemma:V-lattice} The set \lattice{T} forms a bounded join-subsemilattice of $\Fk^3$. 
\end{lemma}

\begin{proof} Being a lattice polynomial, the map $\meansymbol \colon \Pk^3 \to \Pk$ is monotone. It follows that for all $\ABC$, $\ABC['] \in \Pk$, the inclusion
	\begin{equation*}
		\mean{\tuple{A \cup A', B \cup B',C \cup C'}} \supseteq \mean{\ABC} \cup \mean{\ABC[']}
	\end{equation*}
	holds, whence also
	\begin{align*}
	& \big( C \cup C' \big) \setminus \mean{\tuple{A \cup A',B \cup B',C \cup C'}} \\ 
	\subseteq & \big( C \cup C' \big) \setminus \big( \mean{\ABC} \cup \mean{\ABC[']} \big) \\ 
	\subseteq & \big( C \setminus \mean{\ABC} \big) \cup \big( C' \setminus \mean{\ABC[']} \big).
	\end{align*}		 
   Thus if both $C \setminus \mean{\ABC}$ and  $C' \setminus \mean{\ABC[']}$ are finite, then $(C \cup C') \setminus \mean{\tuple{A \cup A',B \cup B',C \cup C'}}$ 
	is finite as well. It follows that $\lattice{T}$ is join-subsemilattice of $\Fk^3$. Finally, it is clear that both $0_{\Fk^3} = \tuple{\emptyset,\emptyset,\emptyset}$ and $1_{\Fk^3} = \tuple{\kappa,\kappa,\kappa}$ belong to $\lattice{T}$.  
\end{proof}

 Let $\lattice{S} := \lattice{T} \cap \lattice{M}_3[\Fk]$ denote the set of all balanced triples from $\lattice{T}$.   

\begin{lemma}\label{lemma:Vclosureclosed} The join-semilattice $\lattice{T}$ is closed under the $\closure{\tuple{-}}$ operation. 
\end{lemma}

\begin{proof} Let $\ABC \in \lattice{T}$. Since $\Fk$ is a lattice, we have that all $A \cup \mean{\ABC}$, $B \cup \mean{\ABC}$ and $C \cup \mean{\ABC}$ belong to $\Fk$. Since the map $\meansymbol \colon \Pk^3 \to \Pk$ is monotone, the inclusion $\mean{\ABC} \subseteq \mean{\closure{\ABC}}$ holds. It follows that 
\begin{equation*} 
\big( C \cup \mean{\ABC} \big) \setminus \mean{\closure{\ABC}} \subseteq C \setminus \mean{\ABC},
\end{equation*}
which is finite due to $\ABC$ being an element of \lattice{T}. 
\end{proof}

\begin{lemma}\label{lemma:Usublattice} The set $\lattice{S}$ forms a bounded sublattice of the lattice $\lattice{M}_3[\Fk]$. 
\end{lemma}

\begin{proof} Applying Lemmas~\ref{lemma:V-lattice} and~\ref{lemma:Vclosureclosed}, we deduce that $\lattice{S}$ is a bounded join-subsemilattice of $\lattice{M}_3[\Fk]$. Therefore, it suffices to verify that $\lattice{S}$ is a meet-subsemilattice of $\Fk^3$. It is easy to observe that if at least one of $\ABC, \ABC['] \in \Pk^3$ is balanced, then
\begin{equation*}
\mean{\tuple{A \cap A',B \cap B',C \cap C'}} = \mean{\ABC} \cap \mean{\ABC[']}.
\end{equation*}
>From this we  we get that if $\ABC, \ABC['] \in \lattice{S}$, then
\begin{equation*}\begin{aligned} 
& \big(C \cap C' \big) \setminus \mean{\tuple{A \cap A',B \cap B',C \cap C'}} \\ 
= & \big(C \cap C' \big) \setminus \big(\mean{\ABC} \cap \mean{\ABC[']} \big) \\ 
\subseteq & \big( C \setminus \mean{\ABC} \big) \cup \big( C' \setminus \mean{\ABC[']} \big),
\end{aligned}\end{equation*}
so the set $(C\cap C')\setminus \mean{\tuple{A \cap A'. B \cap B', C \cap C'}}$ is finite. This concludes the proof.
\end{proof}

As discussed in Section~\ref{Basic concepts}, since the lattice $\Fk$ is distributive, the lattice $\lattice{M}_3[\Fk]$ is modular. Observe that the mapping $A \mapsto \tuple{A,A,A}$ embeds $\Fk$ into $\lattice{S}$, from which we deduce that 
\begin{equation*}
\vert \Fk \vert \le \vert \lattice{S} \vert \le \vert \Fk^3 \vert.
\end{equation*}
Since the size of both $\Fk$ and $\Fk^3$ is $\kappa$, we get that $\vert \lattice S \vert = \kappa$. Let us sum up these observations in the following corollary to Lemma~\ref{lemma:Usublattice}.  

\begin{corollary}\label{corollary:Uboundedmodular} For $\kappa$ countable infinite, $\lattice{S}$ forms a countable bounded modular lattice. 
\end{corollary}

\begin{remark} Note that unlike $\lattice{S}$, the lattice $\lattice{T}$ is not a meet-subsemillatice of $\Fk^3$. Indeed, both $\tuple{\kappa,\emptyset,\kappa}, \tuple{\emptyset,\kappa,\kappa} \in \lattice{T}$ while $\tuple{\kappa,\emptyset,\kappa} \wedge \tuple{\emptyset,\kappa,\kappa} = \tuple{\emptyset,\emptyset,\kappa} \notin \lattice{T}$.
\end{remark}

\section{A Banaschewski function on \lattice{S}}\label{Banschewski function on U} 

In this section we define a Banaschewski function $f \colon \lattice{S} \to \lattice{S}$ and describe, element-wise, its range. 

\begin{lemma}\label{lemma:fBanaschewski} The map $f \colon \lattice{S} \to \lattice{S}$ defined by
\begin{equation}\label{equation:definitionoff}
f\ABC := \tuple{\kappa \setminus A, \kappa \setminus (B \cup C), \kappa \setminus (A\cup B\cup C)}, \quad\text{for all}\ \ABC \in \lattice{S},
\end{equation}
is a Banaschewski function on \lattice{S}.
\end{lemma}

\begin{proof} First we prove that $\lattice{S}$ contains the range of the map $f$. Observe that if we put $A' := \kompl A$ and $B' := \kompl {(B\cup C)} $, 
then $f\ABC = \tuple{A', B', A' \cap B'}$. Since \Fk\ is a Boolean lattice, the sets $A', B'$ and $A'\cap B'$ all belong to \Fk. Furthermore, we have that
	\begin{equation*}
		A'\cap B' = \mean{\tuple{A', B', A'\cap B'}} = \mean{f\ABC}.
	\end{equation*}
	In particular, $A'\cap B' \setminus \mean{f\ABC} = \emptyset $, whence $f\ABC \in \lattice{S}$.
	
	It is clear from \eqref{equation:definitionoff} that the map $f$ is antitone. Finally, we check that
\begin{equation*}
 1_{\lattice{S}} = \tuple{\kappa,\kappa,\kappa} = \ABC \oplus f\ABC, \quad\text{for all}\ \ABC \in \lattice{S}.
\end{equation*}
It follows immediately from the definition of $f$ that 
	\begin{equation*}
	\ABC\wedge f\ABC = \tuple{\emptyset, \emptyset, \emptyset} = 0_{\lattice{S}}.
	\end{equation*}
To prove that $\ABC \vee f\ABC = 1_\lattice{S}$, let us verify that	
\begin{equation}\label{equation:f a komplementy}
  \kappa = \mean{\tuple{A \cup (\kompl A), B \cup (\kompl {(B\cup C)}), C \cup (\kompl {(A\cup B\cup C)})}}.
\end{equation}
Note that each element of $\kappa$ that is \emph{not} contained in $C$ belongs to $ B\cup (\kompl {(B\cup C)})$. Together with $A\cup (\kompl A)=\kappa$, we get that \eqref{equation:f a komplementy} holds, which concludes the proof.
\end{proof}   
	
\begin{lemma}\label{lemma:E} Let \lattice{E} denote the range of the Banaschewski function $f \colon \lattice{S} \to \lattice{S}$. Then
\begin{equation*}
 		\lattice{E} = \setof{ \tuple{A, B, A \cap B}}{A, B \in \Fk}
\end{equation*}	
and the mapping 
\begin{equation}\label{Eq:mapping}
\tuple{A,B,A \cap B} \mapsto \tuple{A,B}
\end{equation} 
determines an isomorphism from $\lattice{E}$ onto the Boolean lattice $\Fk \times \Fk$.
\end{lemma}

\begin{proof} While proving Lemma~\ref{lemma:fBanaschewski}, we have observed that
\begin{equation}\label{equation:popis E}
	\lattice{E} \subseteq \setof{\ABC \in \lattice{S}}{C = A \cap B} = \setof{ \tuple{A', B', A'\cap B'}}{A', B' \in \Fk}.
\end{equation}
A straightforward computation gives that $f( f\tuple{A', B', A'\cap B'}) = \tuple{A', B', A'\cap B'}$, so the lattice \lattice{E} is equal to the right-hand side of \eqref{equation:popis E}. Finally, it is readily seen that the correspondence \eqref{Eq:mapping} determines an isomorphism $\lattice{E} \to \Fk \times \Fk$.  
\end{proof}

It was noted in \cite{Weh10} that if the range of a Banaschewski function on a lattice \lattice{L} is Boolean, then it is a \emph{maximal} Boolean sublattice of \lattice{L}. Thus we derive from Theorem~\ref{lemma:E} that \lattice{E} is a maximal Boolean sublattice of \lattice{B}.

\section{The counter-example} 

In the present section, we construct another maximal Boolean sublattice $\lattice{B}$ of the lattice $\lattice{S}$. We show that the lattices $\lattice{B}$ and $\lattice{S}$ are not isomorphic and we prove directly that the lattice $\lattice{B}$ is not the range of any Banaschewski function on \lattice{S}. 

\begin{lemma}\label{lemma:g}
	The assignment $\AC \mapsto g\AC := \tuple{A, A \cap C, C}$ defines a bounded lattice embedding $g \colon \Fk \times \Fk \to \lattice{M}_3[\Fk]$. In particular, the range of $g$ is a bounded Boolean sublattice of $\lattice{M}_3[\Fk]$ isomorphic to $\Fk \times \Fk$.
\end{lemma}

\begin{proof}
	It is clear from the definition of the map $g$ that it is injective and that its range is included in $\lattice{M}_3[\Fk]$. Further, for any $A, A', C, C'  \subseteq \kappa$, the equality
	\begin{equation*}
	 g\AC \wedge g\AC['] = g \tuple{A \cap A', C \cap C'}
	\end{equation*}
	holds by \eqref{equation:meet}, while
	\begin{equation}\label{equation:JoinInB}
		g\AC \vee g\AC['] = g \tuple{A\cup A', C\cup C'}
	\end{equation}
	can be easily deduced from \eqref{equation:closure} and \eqref{equation:join}.
	Finally, observe that $g\tuple{\kappa, \kappa} = \tuple{\kappa, \kappa, \kappa}$ and $g\tuple{\emptyset, \emptyset} = \tuple{\emptyset, \emptyset, \emptyset}$, which concludes the proof.
\end{proof}

For any $A, C \in \Fk$, we say that $\AC$ is \emph{finite} if both $A$ and $C$ are finite, and we say that $\AC$ is \emph{co-finite} if both $\kompl A$ and $\kompl C$ are finite. Let us write $A \sim C$ if $\AC$ is either finite or co-finite. Note that there are pairs $A, C \in \Fk$ such that $\AC$ is neither finite nor co-finite; namely, $A \sim C$ if and only if the symmetric difference $(A \setminus C) \cup (C \setminus A)$ is finite. 

\begin{lemma}\label{lemma:Asublattice} The set
	\begin{equation*}
   \lattice{A} = \setof{ \AC \in \Fk}{A \sim C}
   \end{equation*} 
	form a bounded Boolean sublattice of $\Fk \times \Fk$.
\end{lemma}

\begin{proof}
	Let $\AC$, $\AC[']$ be a pair of elements from $\lattice{A}$. If at least one of them is finite, then $\tuple{A \cap A', C\cap C'}$ is clearly finite as well. If both $\AC$ and $\AC[']$ are co-finite, then so is $\tuple{A\cap A', C\cap C'}$. In either case, $\tuple{A\cap A', C\cap C'} \in \lattice{A}$.
	
	If at least one of the pairs $\AC, \AC[']$ is co-finite, then $\tuple{A \cup A', C \cup C'}$ is co-finite, while if both $\AC$ and $\AC[']$ are finite, then so is 
	$\tuple{A \cup A', C \cup C'}$. In particular, $\tuple{A \cup A', C\cup C'} \in \lattice{A}$ whenever  $\AC, \AC['] \in \lattice{A}$.
	
	We have shown that $\lattice{A}$ is a sublattice of $\Fk \times \Fk$. To complete the proof, observe that $\tuple{\emptyset,\emptyset}$ is finite and $\tuple{\kappa,\kappa}$ is co-finite and that the unique complement in $\Fk \times \Fk$ of each $\AC \in \lattice{A}$, namely $\tuple{\kompl A, \kompl C}$ belongs to $\lattice{A}$.
\end{proof}

\begin{lemma}\label{lemma:Bsublattice} The $g$-image $\lattice{B} = g(\lattice{A})$ of $\lattice{A}$ is a bounded Boolean sublattice of $\lattice{S}$.  
\end{lemma}

\begin{proof} Due to Lemma~\ref{lemma:g} and Lemma~\ref{lemma:Asublattice}, $\lattice{B}$ is a bounded Boolean sublattice of $\lattice{M}_3[\Fk]$. Thus in view of Lemma~\ref{lemma:Usublattice}, it suffices to verify that $\lattice{B} \subseteq \lattice{S}$, that is, that $C \setminus (A \cap C)$ is finite for every $\AC \in \lattice{A}$. This is clear when $\AC$ is finite. If $\AC$ is co-finite, then $C \setminus (A \cap C) = C \setminus A \subseteq \kappa \setminus A$ is finite and we are done.  
\end{proof}

Observe that if $\ABC$ is a balanced triple then $B \subseteq A$ if and only if $B = A \cap B = A \cap C$. It follows that 

\begin{equation}\label{equation:popisB}	
	\lattice{B} = \setof{ \ABC \in \lattice{S}}{ A \sim C \text{ and } B \subseteq A}.
\end{equation}

\begin{lemma}\label{L:U-B a komplement}
	Let $\ABC \in \lattice{S} \setminus \lattice{B}$ and let $\ABC[']$ be a complement of $\ABC$ in \lattice{S}. If $B \subseteq A$, then $B' \not \subseteq A'$.
\end{lemma}

\begin{proof}
	Since $\ABC \not \in \lattice{B}$ and $B \subseteq A$, it follows from \eqref{equation:popisB} that $A \nsim C$. Hence exactly one of the two sets $A, C$ is finite. From $B \subseteq A$ and $C \setminus B$ being finite we conclude that $C$ and $\kappa \setminus A$ are finite. It follows that the set $B=B\cap C$ is finite as well.
	
	Suppose now that $B'\subseteq A'$. Since $\ABC \wedge \ABC['] = 0_{\lattice{S}}$, we have that $A \cap A' = \emptyset$, whence the set $A'$ $\subseteq \kappa \setminus A$ is finite. A fortiori, the set $B'$ is also finite due to the assumption that $B'\subseteq A'$. As $ C'\setminus B' = C'\setminus (B'\cap A') = C'\setminus \mean{\ABC[']}$ is also finite, we conclude that so is $C'$. But then
	\begin{equation*}
	\mean{\tuple{A\cup A', B\cup B', C\cup C'}} \subseteq B \cup B'\cup C\cup C'
	\end{equation*}
	is a finite set, which contradicts the assumption that $\ABC \vee \ABC['] = \tuple{\kappa, \kappa, \kappa} = 1_\lattice{S}$. 
\end{proof}

\begin{corollary}\label{C:U-B a komplement} Every complemented bounded sublattice $\lattice{C}$ of $\lattice{S}$ such that $\lattice{B} \subsetneq \lattice{C}$ contains an element $\ABC$ with $B \not \subseteq A$.  
\end{corollary}

\begin{proof} Let $\ABC \in \lattice{C} \setminus \lattice{B}$ and let $\ABC[']$ be its complement in \lattice{C}. Applying Lemma~\ref{L:U-B a komplement}, we get that either $B \not\subseteq A$ or $B' \not\subseteq A'$.
\end{proof}

\begin{proposition}\label{proposition:Bmaximal}
	The lattice \lattice{B}\ is a maximal Boolean sublattice of \lattice{S}.
\end{proposition}

\begin{proof}
	Let \lattice{C}\ be a complemented bounded sublattice of \lattice{S}\ satisfying $\lattice{B} \subsetneq \lattice{C}$. There is $\ABC \in \lattice{C}$ with $B \not \subseteq A$ by Corollary~\ref{C:U-B a komplement}. We can pick a finite nonempty $F \subseteq (B \setminus A)$. Since the triple $\ABC$ is balanced, 
	\begin{equation}\label{equation:FC empty}
		\emptyset = F \cap A = F \cap B \cap A = F \cap B \cap C = F \cap C.
	\end{equation}
	Now observe that both $g\tuple{F,\emptyset}$ and $g\tuple{\emptyset,F}$ are in \lattice{B}. Applying \eqref{equation:JoinInB} and \eqref{equation:FC empty}, we get that
\begin{equation}\label{equation:FirstNonDistributive}
	\ABC \wedge \big( g\tuple{F,\emptyset} \vee g\tuple{\emptyset,F} \big) = \ABC \wedge g\tuple{F,F} = \tuple{\emptyset,F,\emptyset}, 
\end{equation}       
while
\begin{equation}\label{equation:SecondNonDistributive}
	\big( \ABC \wedge g\tuple{F,\emptyset} \big) \vee \big( \ABC \wedge g\tuple{\emptyset,F} \big)
	= \tuple{\emptyset,\emptyset,\emptyset}. 
\end{equation} 
It follows from \eqref{equation:FirstNonDistributive} and \eqref{equation:SecondNonDistributive} that 
	the lattice \lattice{C}\ is not distributive, {\em a fortiori} it is not Boolean.
\end{proof}

\begin{proposition}\label{Prop:B not Ban range}
	The sublattice \lattice{B}\ of \lattice{S}\ is not the range of any Banaschewski function on \lattice{S}.
\end{proposition}

\begin{proof}
	The range of a Banaschewski function on \lattice{S}\ must contain a complement of each element of \lattice{S}. We show that no complement of $\tuple{\kappa, \emptyset, \emptyset}$ in \lattice{S}\ belongs to \lattice{B}. 
	
	Suppose the contrary, that is, that there is $\ABC = g\AC \in \lattice{B}$ satisfying $\tuple{\kappa, \emptyset, \emptyset} \oplus \ABC = 1_{\lattice{S}}$. Then $A = A \cap \kappa =\emptyset$, and by \eqref{equation:popisB} also $B = \emptyset$. Then from $B = \emptyset$ and  $\tuple{\kappa, \emptyset, \emptyset} \vee \ABC = 1_\lattice{S}$, one infers that $C = \kappa$. It follows that $\ABC \notin \lattice{S}$; indeed, $C \setminus \mean{\ABC} = C \setminus \emptyset = \kappa$ is not finite. Thus $\ABC \not \in \lattice{B}$, which is a contradiction.
\end{proof}

\begin{remark}
	Note that for the particular case of $\kappa=\aleph_0$, the assertion of Proposition~\ref{Prop:B not Ban range} follows from Proposition~\ref{Prop:B E noniso} together with \cite[Corollary 4.8]{Weh10}, which states that the ranges of two Boolean Banaschewski functions on a \emph{countable} complemented modular lattice are isomorphic.
\end{remark}

\begin{proposition}\label{Prop:B E noniso}
	The lattices \lattice{B}\ and \lattice{E}\ are not isomorphic.
\end{proposition}

\begin{proof}
	In \lattice{B}, every \emph{finite} element $g\AC$ is a join of a finite set of atoms, namely
	\begin{equation*}
		g\AC = \left(
		\bigvee_{\alpha \in A} g\tuple{\{\alpha\}, \emptyset} 
		\right) \vee \left(
		\bigvee_{\gamma \in C} g\tuple{\emptyset, \{\gamma\}} 
		\right),
	\end{equation*}	
and, dually, every co-finite element is a meet of a finite set of co-atoms. On the other hand, there are elements in $\Fk \times \Fk$ that are neither finite joins of atoms nor finite meets of co-atoms. Recall that in Lemma~\ref{lemma:E}, we have observed that the lattice \lattice{E}\ is isomorphic to $\Fk\times \Fk$. Therefore the lattices \lattice{B}\ and \lattice{E}\ are not isomorphic.
\end{proof}

\section{Representing \lattice{S} in a subspace-lattice} 

Although the construction in the three previous sections was performed for an \emph{infinite} cardinal $\kappa$, the results of the present section on embedding the lattice $\lattice M_3[\Pk]$ into $\Sub (\lattice V)$ (namely Theorem~\ref{theorem:U->Sub}) work just as well for $\kappa$ finite.
In particular, Proposition~\ref{prop:M3L Arguesian} (an enhancement of \cite[Lemma 2.9]{GW99b}) holds for lattices of any cardinality.

Let $\mathbb{F}$ be an arbitrary field and let $\vectorspace{V}$ denote the vector space over the field $\mathbb{F}$ presented by generators $x_\alpha, y_\alpha, z_\alpha$, $\alpha \in \kappa$, and relations $x_\alpha + y_\alpha + z_\alpha = 0$. For a subset $X$ of the vector space $\vectorspace{V}$ we denote by $\Span(X)$  the subspace of $\vectorspace{V}$ generated by $X$. Given subspaces of $\vectorspace{V}$, say $\vectorspace{X}$ and $\vectorspace{Y}$, we will use the notation $\vectorspace{X} + \vectorspace{Y} = \Span(\vectorspace{X} \cup \vectorspace{Y})$. Let $\Sub(\vectorspace{V})$ denote the lattice of all subspaces of the vector space $\vectorspace{V}$. 

For all $A,B,C \subseteq \kappa$ we put $\vectorspace{X}_A = \Span(\setof{x_\alpha}{\alpha \in A})$, $\vectorspace{Y}_B = \Span(\setof{y_\beta}{\beta \in B})$, and $\vectorspace{Z}_C = \Span(\{z_\gamma \mid \gamma \in C\})$. 

We define a map $F \colon \Pk^3 \to \Sub(\vectorspace{V})$ by the correspondence 
\begin{equation}\label{equation:definitionofF}
\ABC \mapsto \vectorspace{X}_A + \vectorspace{Y}_B + \vectorspace{Z}_C. 
\end{equation}
Each of the sets $\setof{x_\alpha}{\alpha \in \kappa}$, $\setof{y_\beta}{\beta \in \kappa}$, and $\setof{z_\gamma}{\gamma \in \kappa}$ is clearly linearly independent.	It follows that $\vectorspace{X}_{A \cup A'} = \vectorspace{X}_A + \vectorspace{X}_{A'}$ for all $A,A'\subseteq \kappa$ and, similarly,  $\vectorspace{Y}_{B \cup B'} = \vectorspace{Y}_B + \vectorspace{Y}_{B'}$ and $\vectorspace{Z}_{C \cup C'} = \vectorspace{Z}_C + \vectorspace{Z}_{C'}$ for all $B,B',C,C'\subseteq \kappa$. A straightforward computation gives the following lemma:   

\begin{lemma}\label{lemma:Fjoinhomomorphism} The map  $F \colon \Pk^3 \to \Sub(\vectorspace{V})$ is a bounded join-homomorphism. 
\end{lemma}

\begin{proof} Clearly $F\tuple{\emptyset,\emptyset,\emptyset} = \vectorspace{0}$ and $F\tuple{\kappa,\kappa,\kappa} = \vectorspace{V}$. Following the definitions, we compute $F(\ABC) + F(\ABC[']) = \vectorspace{X}_{A} + \vectorspace{Y}_{B} + \vectorspace{Z}_{C} + \vectorspace{X}_{A'} + \vectorspace{Y}_{B'} + \vectorspace{Z}_{C'}  
= \vectorspace{X}_{A \cup A'} + \vectorspace{Y}_{B \cup B'} + \vectorspace{Z}_{C \cup C'} = F(\tuple{A \cup A', B \cup B', C \cup C'}).$
\end{proof}

Let $G \colon \Sub(\vectorspace{V}) \to \Pk^3$ be a map defined by 
\begin{equation*}
\vectorspace{W} \mapsto \tuple{\setof{\alpha}{x_\alpha \in \vectorspace{W}},\setof{\beta}{y_\beta \in \vectorspace{W}}, \setof{\gamma}{z_\gamma \in \vectorspace{W}}},
\end{equation*}
for all $\vectorspace{W} \in \Sub(\vectorspace{V})$.

It is straightforward that $G$ is a bounded meet-homomorphism and that it is the right adjoint of $F$ (i.e., replacing the lattice $\Sub(\vectorspace{V})$ with its dual, the maps $F$ and $G$ form a Galois correspondence \cite{Ore44}). Indeed, one readily sees that 
\begin{equation*}
F\ABC \subseteq \vectorspace{W} \text{ iff } \ABC \le G(\vectorspace{W}).
\end{equation*}  

The maps $F$ and $G$ induce a closure operator $GF$ on $\Pk^3$. 

\begin{lemma}\label{lemma:GF=closure} The composition $GF \colon \Pk^3 \to \Pk^3$ is precisely the closure operator $\closure{\tuple{-}}$ on $\Pk^3$ defined by \eqref{equation:closure}.
\end{lemma}

\begin{proof} We shall prove that $GF\ABC = \closure{\ABC}$ for every $\ABC \in \Pk^3$. By symmetry, it suffices to prove that 
\begin{equation*}
\setof{\alpha \in \kappa}{x_\alpha \in F\ABC} = A \cup (B \cap C).
\end{equation*} 
Let $\alpha \in A \cup (B \cap C)$. If $\alpha \in A$, then $x_\alpha \in F\ABC$ by the definition \eqref{equation:definitionofF}, while if $\alpha \in B \cap C$, then $x_\alpha = - y_\alpha - z_\alpha \in F\ABC$ by \eqref{equation:definitionofF} and the defining relations of $\vectorspace{V}$. It follows that $A \cup (B \cap C) \subseteq \setof{\alpha \in \kappa}{x_\alpha \in F\ABC}$.   
 \par In order to prove the opposite inclusion, take any $\xi \in \kappa \setminus A$ satisfying $x_\xi \in F\ABC$; if there is one, there is nothing to prove. We need to show that then $\xi \in B\cap C$. Certainly
\begin{equation}\label{equation:XYZcoordinates}
x_\xi = \sum_{\alpha \in A} a_\alpha x_\alpha + \sum_{\beta \in B} b_\beta y_\beta + \sum_{\gamma \in C} c_\gamma z_\gamma
\end{equation}  
for suitable $a_\alpha$, $b_\beta$, and $c_\gamma \in \mathbb{F}$ such that all but finitely many of them are zero. We set $a_\alpha = 0$ for $\alpha \notin A$, $b_\beta = 0$ for $\beta \notin B$, and $c_\gamma = 0$ for $\gamma \notin C$. Since $z_\gamma + x_\gamma + y_\gamma = 0$ for every $\gamma \in \kappa$, it follows from \eqref{equation:XYZcoordinates} that
\begin{equation}\label{equation:XYcoordinates}
x_\xi = \left(\sum_{\alpha \in A} a_\alpha x_\alpha - \sum_{\gamma \in C} c_\gamma x_\gamma \right) + \left( \sum_{\beta \in B} b_\beta y_\beta - \sum_{\gamma \in C} c_\gamma y_\gamma \right).
\end{equation}  
It easily follows from the defining relations of $\vectorspace{V}$ that $\setof{x_\alpha,y_\alpha}{\alpha \in \kappa}$ forms a basis of $\vectorspace{V}$. Thus, applying 
\eqref{equation:XYcoordinates} we get that
\begin{equation}\label{equation:abcformulas}
a_\xi - c_\xi = 1 \text{ and } b_\xi - c_\xi = 0.
\end{equation}
Since by our assumption $\xi \notin A$, we get from \eqref{equation:XYZcoordinates} that $a_\xi = 0$. Substituting to \eqref{equation:abcformulas} we get that $b_\xi = c_\xi = - 1$, hence $\xi \in B \cap C$. This concludes the proof that $A\cup(B\cap C)\supseteq \set{\alpha \in \kappa | x_\alpha \in F\ABC}$.
\end{proof}

The next lemma shows that $F \restriction \lattice{M}_3[\Pk]$ preserves meets. Note that with Lemma~\ref{lemma:Fjoinhomomorphism}, this means that $F \restriction \lattice{M}_3[\Pk]$ is a lattice embedding of $\lattice{M}_3[\Pk]$ into the lattice $\Sub(\vectorspace{V})$.  

\begin{lemma}\label{lemma:Fpreservesmeets} Let $\ABC, \ABC['] \in \lattice{M}_3[\Pk]$ be balanced triples. Then 
\begin{equation*}
F \ABC \cap F \ABC['] = F \tuple{A \cap A', B \cap B', C \cap C'}.
\end{equation*}
\end{lemma}

\begin{proof} Since, by Lemma~\ref{lemma:Fjoinhomomorphism}, $F$ is a join-homomorphism, it is monotone, whence $F\tuple{A \cap A', B \cap B', C \cap C'} \subseteq F\ABC \cap F\ABC[']$. Thus it remains to prove the opposite inclusion. 
\par
Let $v \in F\ABC \cap F\ABC[']$ be a non-zero vector. Then $v$ can be expressed as
\begin{equation}\label{equation:expression}
v = \sum_{\alpha \in A} a_\alpha x_\alpha + \sum_{\beta \in B} b_\beta y_\beta + \sum_{\gamma \in C} c_\gamma z_\gamma =
\sum_{\alpha \in A'} a_\alpha' x_\alpha + \sum_{\beta \in B'} b_\beta' y_\beta + \sum_{\gamma \in C'} c_\gamma' z_\gamma. 
\end{equation}
Consider such an expression of $v$ with 
\begin{equation}\label{equation:sizeofexpression}
\vert \setof{\alpha}{a_\alpha \neq 0} \vert + \vert \setof{\beta}{b_\beta \neq 0} \vert + \vert \setof{\gamma}{c_\gamma \neq 0} \vert
\end{equation}
minimal possible. Put $a_\alpha = 0$ for $\alpha \notin A$, $b_\beta = 0$ for $\beta \notin B$, and $c_\gamma = 0$ for $\gamma \notin C$. By symmetry, we can assume that $a_\alpha \neq 0$ for some $\alpha \in A$. Suppose for a contradiction that $\alpha \notin A'$. Since the triple $\ABC[']$ is balanced, $B' \cap C' \subseteq A'$, whence $\alpha \notin B' \cap C'$. Without loss of generality we can assume that $\alpha \notin B'$. If all $a_\alpha, b_\alpha$, and $c_\alpha$ were non-zero, we could replace $c_\alpha z_\alpha$ with $ - c_\alpha x_\alpha - c_\alpha y_\alpha$ and reduce the value of the expression in \eqref{equation:sizeofexpression} which is assumed minimal possible. Thus either $b_\alpha = 0$ or $c_\alpha = 0$ (recall that we assume that $a_\alpha \neq 0$). We will deal with these two cases separately. If $b_\alpha = 0$, then the equality
\begin{equation}\label{equation:case1}
a_\alpha x_\alpha + c_\alpha z_\alpha = c_\alpha' z_\alpha
\end{equation}
must hold true. Since $x_\alpha$ and $z_\alpha$ are linearly independent, it follows from \eqref{equation:case1} that $a_\alpha = 0$ which contradicts our choice of $\alpha$. The remaining case is when $c_\alpha = 0$. Under this assumption we have that
\begin{equation*}
a_\alpha x_\alpha + b_\alpha y_\alpha = c_\alpha' z_\alpha.
\end{equation*}
It follows that 
\begin{equation}\label{equation:case2b}
a_\alpha x_\alpha = c_\alpha' z_\alpha - b_\alpha y_\alpha = - c_\alpha' x_\alpha - (c_\alpha' + b_\alpha) y_\alpha.
\end{equation}  
Since $x_\alpha$ and $y_\alpha$ are linearly independent, we infer from \eqref{equation:case2b} that $a_\alpha = - c_\alpha' = b_\alpha$. Then 
we could reduce the value of \eqref{equation:sizeofexpression} by replacing $a_\alpha x_\alpha + b_\alpha y_\alpha$ with $c_\alpha' z_\alpha$ in \eqref{equation:expression}. This contradicts the minimality of \eqref{equation:sizeofexpression}.  
\end{proof}

Combining Lemma~\ref{lemma:Fjoinhomomorphism}, Lemma~\ref{lemma:GF=closure}, and Lemma~\ref{lemma:Fpreservesmeets}, we conclude:

\begin{theorem}\label{theorem:U->Sub} The restrictions $F \restriction \lattice{M}_3[\Pk] \colon \lattice{M}_3[\Pk] \to \Sub(\vectorspace{V})$ and, \emph{a fortiory}, $F \restriction \lattice{S} \colon \lattice{S} \to \Sub(\vectorspace{V})$ are bounded lattice embeddings. In particular, the lattice $\lattice{S}$ is isomorphic to 
a bounded sublattice of the subspace-lattice of a vector space.  
\end{theorem}

It is well-known that a distributive lattice $\lattice L$ embeds (via a bounds-preserving lattice embedding) into the lattice $\Pk$, where $\kappa$ is the cardinality of the set of all maximal ideals of $\lattice L$. Such embedding induces an embedding $\lattice{M}_3[\lattice{L}]\hookrightarrow \lattice{M}_3[\Pk]$ (cf. Lemma~\ref{lemma:Usublattice}). By Theorem~\ref{theorem:U->Sub}, the lattice $\lattice{M}_3[\Pk]$ embeds into the lattice $\Sub(\vectorspace{V})$ for a suitable vector space $\vectorspace{V}$ (note again that we now also admit finite $\kappa$). Since the lattice $\Sub (\lattice V)$ is Arguesian, so are $\lattice M_3[\Pk]$ and $\lattice M_3[\lattice L]$. 

On the other hand, \cite[Lemma 2.9]{GW99b} states that a lattice $L$ is distributive if and only if $\lattice M_3[\lattice L]$ is modular. Hence, if $\lattice M_3[\lattice L]$ is modular, it follows that $\lattice L$ is distributive, and, by the above argument, $\lattice M_3[\lattice L]$ is even Arguesian. We have thus proven the following strengthening of \cite[Lemma~2.9]{GW99b}:
\begin{proposition}\label{prop:M3L Arguesian}
	Let $L$ be a lattice. Then $L$ is distributive iff the lattice $\lattice M_3[\lattice L]$ is modular iff $\lattice M_3[\lattice L]$ is Arguesian. If this is the case, then $\lattice M_3[\lattice L]$ can be embedded into the lattice of all subspaces of a suitable vector space over any given field.
\end{proposition}

\bibliographystyle{amsplain}
\bibliography{myabbrev,bibwehrung,bibother}

\end{document}